\documentclass{article}
\usepackage{cite}

\usepackage{amssymb,amsfonts,euscript,mathrsfs}
\usepackage{amsmath, amsthm}

\newtheorem{theorem}{Theorem}
\newtheorem{prop}{Proposition}
\newtheorem{definition}{Definition}

\newtheorem{remark}{Remark}
\newtheorem{corr}{Corollary}

\newcommand{\nstar}{N^{*}}
\begin{document}

\title{Cohomology of the Iwasawa subgroup of the group $U(p,p)$ in nonunitary representations}

\author{A. M. Vershik\thanks{St.~Petersburg Department of Steklov Mathematical Institute; St.~Petersburg State University, St.~Petersburg, Russia. E-mail: {\tt аvershik@gmail.com}. Supported by the RSF grant  14-11-00581.}
 \and 
M. I. Graev\thanks{Institute for System Analysis, Moscow, Russia. E-mail: {\tt mgraev\_36@mtu-net.ru}. Supported by the RFBR grant 13-01-00190-a.}}

%\keywords{подгруппа Ивасавы, коцикл, особое представление}   

\date{September 14,  2015}

\maketitle

\section{Introduction}

This article links up with our paper \cite{1}, where the same problem as here was considered, but for the group
 $U(2,2)$. The main goal is to construct a so-called special representation of the group, i.e., a representation for which the 1-cohomology of the group with values in this representation is nontrivial. We consider the maximal solvable subgroup of
 $U(p,p)$, the Iwasawa subgroup. For commutative and nilpotent groups, the 1-cohomology with values in irreducible unitary representations other than one-dimensional ones is trivial. In particular, it is not clear which solvable groups have irreducible special injective unitary representations. Even in the case of solvable Lie groups we are interested in, the situation seems to be not so simple. The Iwasawa subgroup of     $U(p,p)$ is an extension of the commutative (additive) group of Hermitian matrices of order $p$ and the group of lower triangular matrices of order $p$ with positive diagonal elements, which acts by conjugations on the normal subgroup.

The result of the paper (Theorem~1) is a construction of a special injective operator-irreducible {\it nonunitary bounded representation} of the Iwasawa subgroup of $U(p,p)$ for $p>1$. If ${p=1}$, there exists
a special injective {\it unitary representation}; this is well known, since 
$SU(1,1)$ is a group of rank~1 
(isomorphic to $SL(2,\mathbb{R})$). Whether such a representation exists for
$p>1$ is unknown. The constructed nonunitary special representation is bounded, i.e., all operators of the representation are bounded, which allows one to study it by methods similar to those used in studying unitary representations. Presumably, as in the case   $p=2$ (see \cite{1}), this representation can be extended to an unbounded representation of the whole group 
$U(p,p)$, but this unboundedness can also be in a sense controlled. The importance of special representations (unitary or not) comes from the fact that they can be used to construct, by analogy with the theory for groups of rank~1, representations of groups of currents with values in the corresponding semisimple groups  $U(p,p)$. This  will be considered in another paper.

\section{Background definitions}

\subsection{The Iwasawa subgroup}
Consider the Iwasawa subgroup (= maximal solvable subgroup) 
$P$ of the simple Lie group $U(p,p)$. It can be written as the semidirect product
\begin{align*}
P = S \rtimes N,
\end{align*}
where $S$ is the solvable group of lower triangular matrices of order $p$ with entries
 $s_{ij}$ such that $s_{ii}>0$   and $s_{ij} \in \mathbb{C}$ for $i>j$. From a geometric point of view, $S$ is the product of the complex vector space ${\mathbb C}^{\frac{p(p-1)}{2}}$ and the real octant of dimension~$p$; the group $S$ is not unimodular, but, using geometry, one can easily describe the left and right Haar measures on $S$. The commutative group 
$N$ is the additive group of skew-Hermitian matrices of order~$p$. The real dimensions of these subgroups are equal to 
 $p^2$. Elements of the group $P$ are pairs $(s, n)$, $s \in S$, $n \in N$. In this notation, the product of elements of $P$ is given by the following formula:
\begin{align*}
(s_1, n_1) \cdot (s_2, n_2) = (s_1s_2, s_2^{-1}n_1s_2^{*-1}+n_2).
\end{align*}

For $p=1$, the group $P$ is the two-dimensional subgroup in $GL(2,\mathbb{C})$  of lower triangular $2\times2$ matrices  with equal positive diagonal entries and an imaginary entry under the diagonal.

But we will use only the above definition of the group $P$,  its matrix realization is irrelevant for us here. Our aim is to find  faithful representations of $P$ with nontrivial 1-cohomology. For a representation $\pi$ in a vector space $H$, this means that there exists a $\pi$-valued 1-cocycle, i.e., a map
 $\beta: P\rightarrow H$ such that
 \begin{align}\label{eq:1}
\beta(g_1 \cdot g_2) = \beta(g_1) + T(g_1)\beta(g_2), 
\end{align}
that is not cohomologous to zero (i.e., does not have the form $T_gf-f$).
In the paper, we describe a family of such nonunitary representations.

\subsection{The Hilbert space $H$}
Denote by $\nstar$ the dual group (the group of additive complex characters) of $N$, with elements $m$ and the pairing
\begin{align*}
\langle n, m\rangle = \mathrm{Tr} \,(nm) \in \mathbb{R}.
\end{align*}
We fix the inner product and realize $\nstar$ as a group isomorphic to $N$. The group $S$ acts on
 $\nstar$ by the automorphisms $m \mapsto s^{*}ms$. 

Consider the Hilbert space
\begin{align*}
H = L^2(\nstar, d\nu(m)),
\end{align*}
where $d\nu(m)$ is the Lebesgue measure on $\nstar$, which is only quasi-invariant with respect to the action of $S$:
$$ 
d(s^*ms)=\theta^{2p}(s) \,dm, \quad 
\theta(s) = s_{11}\cdot\ldots \cdot s_{pp} .
$$

\subsection{The space $\mathcal{H}$}

On the group $\nstar$ there are $2^p$ orbits of the group $S$ of maximal dimension
$\dim(S) = p^2$, and their union is complete 
 in the space $\nstar$. These orbits are $\nstar_{\varepsilon} = \{s^{*}\varepsilon s\}$, where $\varepsilon = 
(\varepsilon_1, \ldots, \varepsilon_p)$, $\varepsilon_i = \pm 1$. 
The set of all these orbits $\nstar_{\varepsilon}$ is of full  measure  in  $\nstar$ with respect to $d\nu(m)$, hence the Hilbert space ${H}$ can be written as the direct sum of the subspaces  ${H}_{\varepsilon}$ of functions $f(m)$ concentrated on $\nstar_{\varepsilon}$: 
\begin{align*}
{H} =  \bigoplus_{\varepsilon} {H}_{\varepsilon}.
\end{align*}

\begin{definition}
Denote by $\nstar_0$ the orbit $\{is^{*}s \mid s \in S\}$, and let
$\mathcal{H}$ be the Hilbert subspace of functions concentrated on this orbit, with the norm
$\| f\|^2=\int\limits_{N_0^* }|f(m)|\,d\nu (m)$.  
\end{definition}

\begin{remark}
The remaining $S$-orbits of maximal dimension are $\nstar_{\varepsilon} = 
\{s^{*}\varepsilon s\}$, where $\varepsilon = (\varepsilon_1, \dots, 
\varepsilon_p)$, $\varepsilon_i = \pm 1$. 
\end{remark}

There exists (in general, not measure-preserving) isomorphism of the spaces
$H_{\varepsilon}$ commuting with the action of the group $S$ on these spaces.

The representations (of the groups $N,S,P$) considered below, which are defined on the whole space $H$, induce the corresponding representations on the space
$\mathcal{H}$, given by the same formulas. Hence in what follows we will consider representations only on the space  $\mathcal{H}$.

One can construct another realization of the space $\mathcal{H}$ in terms of the group $S$. To this end, note that the map
 $S \longrightarrow \nstar$ of the form $s 
\mapsto  is^{*}s=m \in  \nstar_0$ is a bijection commuting with the right multiplication on the group $S$ and the action of $S$ on $\nstar_0$.\footnote{This follows from the fact that a positive definite Hermitian matrix can be uniquely written as a product of a lower triangular matrix with positive diagonal entries  and the conjugate matrix.}
In view of this isomorphism,   $\mathcal{H}$ can be realized as the Hilbert space
\begin{align*}
\mathcal{H} = L^2(S, d\widetilde{\nu(s)}),
\end{align*}
where $d\widetilde{\nu(s)}$  is the image of the measure $d\nu(m)$ under the bijection $\nstar_0 
\longrightarrow S$, i.e., as the space of functions on $S$ with the norm
$||f||^2 = \int\limits_S |f(s)|d\widetilde{\nu(s)}$. In this new realization, the operators of the representation are given by the following formula:
\begin{align*}
T(n)f(s) = \exp(i\text{Tr}(sns^{*})f(s);
\end{align*}
thus, just as in the first realization, they act fiberwise. In what follows, we do not use this realization.

\subsection{The representations of the groups $S$ and $P$ in the space $\mathcal{H}$.}

A unitary representation of the group $N$ in the space $\mathcal{H}$ is defined by the formula
\begin{align*}
T(n)f(m) = \exp(i\mathrm{Tr} \,(nm))f(m).
\end{align*}
There are many ways to define a representation of the group $S$ on
$\nstar_0$, namely,
$$
(T(s)f)(m)= \gamma(s,m)f(sms^*),
$$
where $\gamma(s,m)$ is an arbitrary {\it multiplicative cocycle of $S$} with values in the space of invertible functions on $\nstar_0$. The standard choice of this cocycle is the density of the image of the measure $\nu$  under the action of an element $s$ with respect to the original measure. But we will choose a multiplicative cocycle of the group $S$ later, and now let
$a(s)$ be a positive function on $S$. At the moment, we formally define a representation of $S$ by the following formulas:
\begin{align}\label{eq:2}
T_a(s_0)f(m) = \frac{a(s^{*}_0ms_0)}{a(m)}f(s_0^{*}ms_0).
\end{align}
One can easily see that these operators determine a formal representation of the group $S$ and that together with the operators of the subgroup $N$ they generate a formal representation of the whole group~$P$.

We have
\begin{align*}
||T_a(s_0)f||^2 = \int\limits_{\nstar_0} \lvert f(m)b(m, s_0)\rvert^2d\nu(m),
\end{align*}
где
\begin{align*}
b(m, s_0) = 
\frac{a(s^{*}_0ms_0)}{a(m)}\left(\frac{d\nu(s_0^{*-1}ms_0^{-1})}{d\nu(m)}\right)^{1/2}.
\end{align*}

\begin{corr}
{\rm 1)} If $b(m, s_0) \equiv 1$, then the operator  $T_a$ is unitary.

{\rm 2)} If $|b(m,s_0)| < c$, then $||T_a(s_0)f|| < \infty$.

{\rm 3)} If none of these conditions is satisfied, then the operator $T_a$ 
is unbounded.
\end{corr}

One can easily see that the operators $T(n)$ and $T_a(s_0)$ generate, for every fixed $a$, a representation of the whole group $P$. Analogous representations can be defined in a similar way on all orbits of the group~$S$.

\begin{prop}
The constructed representations in the spaces $H_{\epsilon}$ of functions on orbits of the group $S$ are operator-irreducible and pairwise nonequivalent.
\end{prop}

Indeed, an operator commuting with the operators of the representation of the group $N$ must be a multiplicator, since the operators $T(n)$ generate a maximal commutative subalgebra (the algebra of all multiplicators); and a multiplicator commuting with all operators $T_a(\cdot )$ must be constant, since the action of $S$ on an orbit is transitive, and hence ergodic, by the local compactness of the orbit. Representations on different orbits are, obviously, nonequivalent, since the orbits are disjoint.

\section{Nontrivial cohomology of a commutative group with values in the regular representation}

\subsection{The general construction}
As is well known, a commutative locally compact group has no nonidentity unitary representations with nontrivial 1-cohomology; however, the regular representation in $L^2(G)$, which weakly contains the identity representation, does have nontrivial 1-co\-ho\-mo\-lo\-gy. One should take a function that is locally square summable but does not belong to 
$L^2(G)$ because it is not integrable at infinity; note that the difference of such a function with any its shift does belong to
$L^2(G)$. The dual and, of course, equivalent approach adopted below is to realize the regular representation in
$L^2(\widehat G)$, i.e., in the group of characters, by multiplicators, and then find a function such that its singularity at zero (at the identity character) is not square integrable, but the multiplication of this function by
 $1-\langle \chi,g\rangle$ brings it back to
$L^2(\widehat G)$ for every character $\chi$. It is the latter idea that is implemented below for the group $N$. Of course, the choice of the function is not unique; moreover, there is a continuum of pairwise noncohomologous cocycles. Thus the groups
 $H^1(G;\mathbb R)$ and  $H^1(G;\mathbb T)$  are infinite-dimensional. Worse yet, the group of cocycles cohomologous to zero is dense in the group of all cocycles, hence the cohomology group has no reasonable structure. By contrast to this, note that for semisimple groups of rank~1, the groups $H^1(G;\mathbb R)$ are finite-dimensional and nonzero.

\subsection{A special representation of the group $N$}
Consider the space $H = L^2(\nstar, dm)$, where $dm$ is the Lebesgue measure on $\nstar$, in which we have defined the unitary representation
$T$. We will prove that there exists a function 
 $f_0(m)$ such that $f_0$ does not belong to $H$, but the function $\beta(n) = T(n)f_0 - f_0 $ does belong to $H$, i.e.,
$||\beta(n)|| < \infty$ for every $n \in N$. It will follow that  $\beta$ is a nontrivial 1-cocycle  of the representation $T$, i.e., the representation $T$ is special.

Put $|m|^2 = \mathrm{Tr} \, (mm^{*})$ and define a function $\beta(n)$ 
by the following formula:
\begin{align*}
\beta(n) = T(n)f_0 - f_0 \quad  \text{ where } \quad f_0(m) = \frac{e^{-|m|}}{|m|^{p^2/2}}.
\end{align*}

Let us check that the function $f_0$ satisfies the required conditions. To this end, consider spherical coordinates on $\nstar$; namely, given a vector $m$, its spherical coordinates are 
 $r = 
|m|$ and $\omega = r^{-1}m$. The set of elements $\omega \in \nstar$ 
such that  $|\omega| = 1$ will be called the sphere of $\nstar$ and denoted by $\Omega$. Thus every element $m \in \nstar$  
can be uniquely written in the form 
\begin{align*}
m = r\omega \quad \text{ where } r = |m|\quad  \text{ and }  \quad \omega \in \Omega.
\end{align*}

In these spherical coordinates, we have $dm = dr^{p^2-1}drd\omega$, where
$d\omega$ is a measure %???
on the sphere $\Omega$, whence 
\begin{align*}
||f||^2 = \int\limits_{} |f(r\omega)|^2r^{p^2-1}drd\omega .
\end{align*}

In particular, we have
\begin{align*}
||f_0||^2 = e^{-2r}r^{-1}dr d\omega.
\end{align*}

Now it is obvious that $||f_0|| = \infty$, i.e., $f_0 \not\in H$. Let us check that
$||\beta(n)|| < \infty$. In polar coordinates, we have
\begin{align*}
||\beta(n)||^2 = \int\limits_{} |\exp(-ri\mathrm{Tr} \, (n\omega)) - 1|^2 
\exp(-2r)r^{-1}drd\omega ;
\end{align*}
since the expression under the absolute value vanishes for $r = 0$, the integral in $r$ converges, and hence the whole integral converges. Thus we have proved that
 $||\beta(n)|| < \infty$, i.e., $\beta(n)$ is a nontrivial 1-cocycle. The formula for the norm of 
$\beta(n)$ can be simplified by integrating in $r$. According to a well-known formula,
\begin{align}\label{eq:5}
\int\limits_0^{\infty} (e^{-ar} - e^{-br})r^{-1}dr = \log\frac{b}{a},
\end{align}
and we obtain $||\beta(n)||^2 = 
\int\log\left(1+\frac{1}{4}\left(\text{Tr}\left(n\omega\right)\right)^2\right)d\omega$.

\subsection{A special representation of the group $P = S \rtimes N$}
Now we make an appropriate choice of the function $a$, i.e., choose a multiplicative cocycle of the group $S$. Clearly, the representation $T_a$ of the group $S$ defined above generates, together with the original representation of the group $N$, a representation of the whole Iwasawa group~$P$. 

\begin{theorem}
The operators $T_a(s_0)$ on the group $S$ given by the formula
\begin{align*}
T_a(s_0)f(m) = \frac{a(s_0^{*}ms_0)}{a(m)}f(s_0^{*}ms_0),
\end{align*}
where $a(m) = \|m\|^{p^2/2}$, form a bounded special nonunitary representation of the group~$P$.
\end{theorem}

\begin{proof}
As we have proved, $||f_0|| = \infty$, hence it suffices to check that the norm of the cocycle
$\beta(s) = T_a(s)f_0 - f_0$ is finite:
$||\beta(s)|| < \infty$ for every $s \in S$. We use the expression for elements $m$ in polar coordinates
 $m = r\omega$. It follows from the explicit formulas for
  $a$ and $T_a$ that $\dfrac{a(s_0^{*}ms_0)}{a(m)} = ||s_0^{*}\omega 
s_0||^{p^2/2}$ and 
\begin{align*}
||\beta(s_0)||^2 = \int \Big\lvert\,\lvert s_0^{*}\omega 
s_0\rvert^{p^2/2}\frac{e^{-r\lvert s_0^{*}\omega 
s_0\rvert}}{r^{p^2/2}\lvert s_0^{*}\omega 
s_0\rvert}^{p^2/2}-\frac{e^{-r}}{r^{p^2/2}}\Big\rvert^2r^{p^2-1}drd\omega. 
\end{align*}
The obtained expression can be simplified:
\begin{align*}
||\beta(s_0)||^2 = \int \lvert e^{-r \lvert s_0^{*}\omega s_0\rvert} - 
e^{-r}\rvert ^2 r^{-1} dr d\omega. 
\end{align*}
In view of (\ref{eq:5}), integrating in $r$ yields the following formula for the norm of the nontrivial 1-cocycle
 $\beta(s_0)$: 
\begin{align*}
||\beta(s_0)||^2 = \int\limits_{\Omega} \log 
\left(\dfrac{(1+|s_0^{*}\omega s_0|)^2}{4|s_0^{*}\omega s_0|}\right) 
d\omega. 
\end{align*}
The convergence of the integral follows from the fact that the function $|s_0^{*} 
\omega s_0|$ is bounded from below. Thus the cocycle
$\beta(s_0)$ is nontrivial, i.e., the representation $T_a$ of the group $S$ is special.

To prove that this representation  is nonunitary, we use the following observation: the representation
 $T_a$ of the group $S$ is unitary if and only if
\begin{align*}
\frac{a(s^{*}ms)}{a(m)} = \theta^{p^2/2}(s).
\end{align*}

Indeed, according to Corollary~1,   the operators $T_a$ are unitary if and only if $c(m, s_0) \equiv 1$.
This immediately implies the required assertion in the case $p>1$.

In the particular case $p = 1$, we  have $S = \mathbb{R}_{+}^{*}$ and  
 $c(m, s_0)\equiv 1$. Thus the operators $T_a$ form a unitary representation. In other words, for
$p = 1$ the extension of the special representation of the subgroup $N$ to a unitary representation of the whole Iwasawa group is also special.

Finally, note that the operators $T_a(s_0)$ are bounded. Indeed, it follows from the general formula that
\begin{align}\label{eq:6}
\lVert T_a(s_0)f(m)\rVert^2 = \int \lvert f(m)c(m, s_0)\rvert^2\, dm,
\end{align}
where $c(m, s_0) =\dfrac{\left(a\left(m\right)\right)}{a\left(s_0^{*-1}m 
s_0^{-1}\right)}\left(\dfrac{ds_0^{*-1}ms_0^{-1}}{dm}\right)^{1/2}$. 

Hence the boundedness of the operators of the representation for all $p$ follows from the boundedness of the functions
 $c(m, s_0)$.  The theorem is proved.
\end{proof}

\begin{corr}
If $a(m) \not= \lVert m\rVert^{p^2/2}$, then the formula for $\beta(s_0)$ takes the form
\begin{align*}
\lVert\beta(s_0)\rVert^2 = \int \lvert b(\omega)e^{-r\lvert s_0^{*}\omega s_0\rvert} - e^{-r}\rvert^2 r^{-1}drd\omega,
\end{align*}
where $b(\omega) \not\equiv 1$. In this case, the integrand does not vanish at
 $r = 0$, hence the integral in $r$ diverges. Thus
$\beta(s_0)$ is not a well-defined  {\rm 1}-cocycle in the representation space. 
\end{corr}

The question of the existence of special representations of the group $P$ reduces to estimating the norms of cocycles for the subgroups $N$ and $S$. It was proved above that for
$p>1$, the nontrivial 1-cocycle $\beta$ on  $N$ associated with the function $f_0(m) = \frac{e^{-\lvert 
m\rvert}}{\lvert m\rvert^{p^2/2}}$ cannot be extended to a unitary representation
 $T^{0}$ of the subgroup $S$. A similar result holds for a wider class of functions of the form
 $f_0\left(m\right) = 
\frac{u\left(m\right)}{\lvert m\rvert^{p^2/2}}$. We do not dwell on this, since it is not known whether this is true for arbitrary functions. Note that $T^{0}$ is the unitary representation of the group $P$ corresponding to the right-invariant Haar measure on $S$. Thus if our conclusion is true for any functions
$a$ and $f$, this will prove that the group $P$ has no unitary faithful special representations. This question is still open.

\medskip
Translated by N.~V.~Tsilevich.

\bigskip

Vershik A.~M.,  Graev M.~I. Cohomology of the Iwasawa subgroup of the group $U(p,p)$ in nonunitary representations.
\smallskip

We construct a special injective nonunitary bounded irreducible
representation for the Iwasawa subgroup of the semisimple Lie group
$U(p,p)$ with $p>1$.

\end{document}